\documentclass[11pt, a4paper, leqno]{amsart}

\usepackage[T1]{fontenc}
\usepackage[%
  pdfauthor={Denis Bonheure, Silvia Cingolani, and Jean Van Schaftingen},%
  pdftitle={The logarithmic Choquard equation: sharp asymptotics and nondegeneracy of the groundstate}]%
  {hyperref}
\usepackage[%
  backrefs,
  abbrev]%
  {amsrefs}

\usepackage{lmodern}
\usepackage{microtype}
\usepackage{enumerate}

\usepackage{todonotes}

\newtheorem{theorem}{Theorem}
\newtheorem{proposition}{Proposition}[section]

\newtheorem{lemma}[proposition]{Lemma}
\newtheorem{corollary}[proposition]{Corollary}
\theoremstyle{definition}
\newtheorem{remark}{Remark}[section]
\numberwithin{equation}{section}

\newcommand{\N}{{\mathbb N}}
\newcommand{\Z}{{\mathbb Z}}
\newcommand{\R}{{\mathbb R}}
\newcommand{\C}{{\mathbb C}}

\newcommand{\abs}[1]{\lvert #1 \rvert}

\newcommand{\Bigabs}[1]{\Bigl\lvert #1 \Bigr\rvert}
\newcommand{\norm}[1]{\lVert #1 \rVert}
\newcommand{\scalprod}[2]{(#1\vert #2)}

\newcommand{\st}{\;:\;}
\newcommand{\cN}{{\mathcal N}}

\renewcommand{\epsilon}{\varepsilon}

\newcommand{\dif}{\,\mathrm{d}}

\title[The logarithmic Choquard equation]{The logarithmic Choquard equation : sharp asymptotics and nondegeneracy of the groundstate}

\author{Denis Bonheure}
\address{
Denis Bonheure \\  
D{\'e}partement de Math{\'e}matique, Universit{\'e} Libre de Bruxelles\\ 
CP 214, Boulevard du Triomphe\\
B-1050 Bruxelles, Belgium\\
and INRIA -- Team MEPHYSTO.}
\email{denis.bonheure@ulb.ac.be}

\author{Silvia Cingolani}
\address{Silvia Cingolani \\ Politecnico di Bari\\ 
	Dipartimento di Meccanica, Matema\-tica e Management\\ 
	Via E.\ Orabona 4 
	\\ 70125 Bari \\ 
	Italy}
\email{silvia.cingolani@poliba.it}

\author{Jean Van Schaftingen}
\address{Jean Van Schaftingen \\ Universit\'e catholique de Louvain\\ 
Institut de Re\-cher\-che en Math\'ematique et Phy\-sique\\ 
Chemin du Cyclotron 2 
bte L7.01.01\\ 1348 Louvain-la-Neuve \\ 
Belgium}
\email{Jean.VanSchaftingen@UCLouvain.be}

\begin{document}

\subjclass[2010]{35J50, 35Q40}

\keywords{Logarithmic Choquard equation; planar Schr\"odinger--Newton system;  positive ground state solution; asymptotics; nondegeneracy}

\begin{abstract}
We derive the asymptotic decay of the unique positive, radially symmetric  solution  to the logarithmic Choquard equation
$$
- \Delta u + a u = \frac{1}{2 \pi}  \Bigl[\ln \frac{1}{|x|}* \abs{u}^2 \Bigr] \ u \qquad \text{in $\R^2$}
$$
and we establish its \emph{nondegeneracy}.
For the corresponding three-dimensional problem, the nondegeneracy property of the positive ground state to the Choquard equation was proved by E. Lenzmann (Analysis \& PDE, 2009).
\end{abstract}
\maketitle

\section{Introduction}
We consider the nonlocal model equation
 \begin{equation}
 \label{eqChoquardNd}
 - \Delta u  + a u  = \bigl[\Phi_N * |u|^2\bigr] \ u   
 \qquad \hbox{in} \ \mathbb{R}^{N}
\end{equation}
where $a$ is a constant and $\Phi : \mathbb{R}^N \to \mathbb{R}$ is the Newton kernel, that is the fundamental solution of the Laplace equation in $\R^N$, namely
\begin{align*}
&\Phi_N(x) = \frac{\Gamma(\frac{N - 2}{2})}{4 \pi^{N/2} |x|^{N-2}}&
&\text{if $N \geq 3$,}&
&\text{ and }&
&\Phi(x) = \frac{1}{2 \pi} \ln \frac{1}{|x|}& 
\text{if $N=2$}.
\end{align*}

In dimension $N = 3$, the integro-differential equation \eqref{eqChoquardNd} has been introduced to study the quantum physics of \emph{electrons in an ionic crystal} (Pekar's polaron model) \cite{pekar}.
It has later also been proposed as a \emph{coupling of quantum physics with Newtonian gravitation} \citelist{\cite{penrose}\cite{Jones1995newtonian}\cite{Diosi1984}}.
E.\thinspace H.\thinspace Lieb has proved the existence of a unique ground state solution of \eqref{eqChoquardNd} in dimension $N=3$, which is positive and radially symmetric \cite{Lieb1977} (see also \cites{lions,mazhao,CingolaniClappSecchi,MorozVanSchaftingen2013,WangYi}). Successively, E.
\thinspace Lenzmann has shown the nondegeneracy of the unique positive ground state solution to the three-dimensional equation \eqref{eqChoquardNd} \cite{Lenzmann2009}.

In this paper we focus on the planar integro-differential equation corresponding to \eqref{eqChoquardNd}
\begin{equation}
\label{prob2}
- \Delta u + a u = \frac{1}{2 \pi}  \Bigl[\ln \frac{1}{|\cdot|} \ast \abs{u}^2 \Bigr] \ u \qquad \text{in $\R^2$}.
\end{equation}
We refer to it as the logarithmic Choquard equation (or  planar Schr\"o\-dinger--Newton system).

This two-dimensional problem has remained for a long time a quite open field of study.
While Lieb's existence proof has a straightforward extensions to the higher dimensions $N=4$ and $N = 5$ and the existence of finite energy solutions is forbidden for \(N \ge 6\) by a Pohozhaev identity (see for example \citelist{\cite{MorozVanSchaftingen2013}*{Proposition 3.1}\cite{CingolaniSecchiSquassina2010}*{Lemma 2.1}\cite{Menzala1983}*{(2.8)}\cite{GenevVenkov2012}*{(56)}}), the situation is less clear for lower dimensions due to the \emph{lack of positivity of the Coulomb interaction energy term}. 
For \(N = 1\), this difficulty has been overcome recently and the existence of a unique ground state has been shown by solving a minimization problem \cite{choquard.stubbe}.

Back to our planar case \(N = 2\), after numerical studies suggesting the existence of bound states \cite{harrison}*{\S 6}, 
Ph.\thinspace{}Choquard, J.\thinspace{}Stubbe and M.\thinspace{}Vuffray have proved the existence of a unique positive radially symmetric solution to \eqref{prob2} by applying a shooting method to the associated system of two ordinary differential equations \cite{choquard.stubbe.vuffray}. 

In contrast with the higher-dimensional case $N \geq 3$, the applicability of variational methods is not straightforward for \(N = 2\). 
Although \eqref{prob2} has, at least formally, a variational structure related
to the energy functional
$$
u \mapsto I(u)= \frac{1}{2} \int_{\R^2} \bigl(|\nabla u|^2
+ a u^2\bigr) + \frac{1}{8 \pi} \int_{\R^2}\int_{\R^2}
\ln(|x-y|^2)\abs{u (x)}^2 \abs{u (y)}^2\dif x \dif y  
$$
this energy functional is \emph{not well-defined on the natural Sobolev space} $H^1(\R^2)$. 

J.\thinspace{}Stubbe has tackled that problem \cite{Stubbe} by setting a
variational framework for \eqref{prob2} within the functional space 
$$
X:= \Bigl\{u \in H^1(\R^2)\st \int_{\R^2}\ln(1+|x|)\abs{u (x)}^2\dif x < \infty\Bigr\},
$$
endowed with a norm defined for each function \(u \in X\) by
\[
\norm{u}_X^2:= \int_{\R^2} \abs{\nabla u(x)}^2 + \abs{u (x)}^2 \bigl(a + \ln_+ \abs{x}\bigr)\dif x,
\]
where, for each \(s \in (0, +\infty)\)  \(\ln_+ s = (\ln s)_+\). 
This functional \(I\) is well-defined and continuously differentiable on the space \(X\). Critical points $u \in X$ of $I$ are strong solutions in $W^{2,p}(\R^2)$, for all $p \ge 1$,
and classical solutions in $C^2(\R^2)$
of \eqref{prob2}.

Even if $X$ provides a variational framework for \eqref{prob2}, 
some difficulties arise. First, the norm of $X$ \emph{is not 
invariant under translations} whereas the functional $I$ is invariant under
translations of \(\mathbb{R}^2\). Second, the quadratic part of the functional $I$ 
is never coercive on $X$, whatever the value of $a \in \mathbb{R}$.

By using strict rearrangement inequalities, J.\thinspace{}Stubbe has proved that there exists, for any $a \geq 0$, a unique ground state, which is a positive spherically symmetric decreasing function \cite{Stubbe}.
In addition, he proved that there exists a negative number
$a^* < 0$ such that for any $a \in (a^*, 0)$ there
are two ground states with different $L^2$ norm 
and that in the limiting
case $a =a^*$, there is again a unique ground state. T.\thinspace{}Weth and the second author \cite{CW} recently constructed a sequence of solution pairs $(\pm u_n)_{n \in \mathbb{N}}\subset X\) of the equation
\eqref{prob2}  such that $I(u_n) \to \infty$ as $n \to + \infty$. They also provided a variational characterization of the least energy solution. Namely, they proved that the restriction of the functional $I$ to the associated Nehari manifold
 $\cN:= \{u \in X \setminus \{0\}\::\: I'(u)u=0\}$
 attains a global minimum and that every minimizer $u \in \cN$ of $I|_{\cN}$ is a solution of \eqref{prob2} which does not change sign and obeys the variational
characterization
 $$
I(u) = \inf_{u \in X} \sup_{t \in \R} I(tu).
 $$
In addition, the following uniqueness result was proved by T.\thinspace{}Weth and the second author.

\begin{theorem}
  [\cite{CW}*{Theorem 1.3}]%
  \label{sec:introduction-3}%
  For every $a>0$, every positive solution $u \in X$ of \eqref{prob2} is radially symmetric up to translation and strictly decreasing in the distance
  from the symmetry center. 
  Moreover $u$ is unique up to translation in \(\R^2\).
\end{theorem}

Our first result is a description of the asymptotic behaviour of this unique positive solution of the logarithmic Choquard equation \eqref{prob2}. 

\begin{theorem}
\label{theoremAsymptotics}
If \(a > 0\) and if \(u \in X\) is a radially symmetric positive solution of \eqref{prob2}, then there exists \(\mu \in (0, +\infty)\) such that, as $|x|\to\infty$, 
\[
 u (x) = \frac{\bigl(\mu + o (1)\bigr)}
 {\sqrt{\abs{x}} (\ln \abs{x})^{1/4}}
 \exp \biggl(-\sqrt{M} e^{-a/M} \int_{1}^{e^{a/M} \abs{x}} \sqrt{\ln s} \dif s \biggr),
\]
where 
\[
 M = \frac{1}{2 \pi} \int_{\mathbb{R}^2} \abs{u}^2.
\]
\end{theorem}

The integral does not seem to have an explicit asymptotic equivalent at the order \(o (1)\) as \(\abs{x} \to \infty\) in terms of elementary functions; roughly speaking it behaves as 
\[
 \int_{1}^{e^{a/M} \abs{x}} \sqrt{\ln s} \dif s
 = \abs{x} \sqrt{\ln \abs{x}} (1 + o (1)),
\]
as \(\abs{x} \to \infty\). 
This integral can be reexpressed in terms of classical special functions (imaginary error function or Dawson function, see Remark~\ref{remarkSpecialFunctions} below).

We obtain this decay rate by studying the decay rate of solutions to the \emph{linear problem}
\[
  -\Delta u + V u = 0,
\]
when \(V (x) \equiv M \ln \abs{x}\) as \(\abs{x} \to \infty\).



%
%

The asymptotic behaviour of $u$ is a key ingredient to derive the precise description of the kernel of  the linear operator \(\mathcal{L}(u)\) defined by 
\begin{equation}
\label{eq:L(u)}
\mathcal{L}(u)  : \tilde X\to L^2(\R^2) :  \varphi\mapsto -\Delta \varphi + (a - w) \varphi + 2 u \Bigl( \frac{\ln}{2\pi} \ast (u \varphi)\Bigr),
\end{equation}
where 
\begin{equation}
\label{eq:w}
w:\R^2\to\R :x\mapsto \frac{1}{2 \pi} \int_{\R^2} \ln \frac{1}{\abs{x - y}} \abs{u (y)}^2\dif y
\end{equation}
and
\begin{multline}
\label{eq:tildeX}
\tilde X:= \Bigl\{\varphi \in X \st \text{there exists \(f \in L^2 (\R^2)\) such that }\\
\text{for every \(\psi \in C^\infty_c (\R^2)\)} \int_{\R^2} \varphi \mathcal{L} (u)\psi = \int_{\R^2} f \psi \Bigr\}.
\end{multline}
By standard arguments, one easily shows that $\mathcal{L}(u)$ is a self adjoint operator acting on \(L^2(\R^2)\) with domain 
\(\tilde X\). Also, differentiating the equation \eqref{prob2}, it is clear that $\gamma \cdot \nabla u\in \ker \mathcal{L}(u)$ for every $\gamma \in\R^2$. Our main result is the nondegeneracy of the positive solution $u$ of Theorem \ref{sec:introduction-3}. Namely, the kernel of the operator $\mathcal{L}(u)$ is exactly the vector space spanned by the partial derivatives of $u$.

\begin{theorem}\label{theoremNondegenerate} 
If $a>0$ and $u \in X$ is a positive solution of \eqref{prob2}, then 
  \[
  \ker \mathcal{L}(u) =
  \bigl\{ 
  \gamma \cdot \nabla u : \gamma\in\R^2\bigr\}.
  \]
\end{theorem}

The paper is organized as follows. In Section~\ref{sec:symm-uniq-posit}, 
we set up the variational framework and establish useful preliminary estimates. 
In Section~\ref{Asymptotic behaviour of the groundstate solution}, 
we study the asymptotic decay and prove Theorem \ref{theoremAsymptotics}. Section~\ref{sec:nondegeneracy} is devoted to the proof of Theorem~\ref{theoremNondegenerate}. 
Assuming without loss of generality that $u$ is radial, 
we prove, as a first step, the nondegeneracy of the linearized operator $\mathcal{L}(u)$ restricted to the subspace of radial functions of $\tilde X$, 
that is, we show the triviality of its kernel on that subspace. 
As a second step, using the fact that $u$ and $w$ are radial, we describe by means of an angular decomposition, how the operator $\mathcal{L}(u)$ acts on each subspace $\tilde X \cap L^2_k(\R^2, \C)$, where
\begin{multline*}
  L^2_k (\R^2; \C) := \bigl\{ f \in L^2 (\R^2;\C) \st \text{for almost every \(z \in \R^2 \simeq \C\)
	  and \(\theta \in \R\), }\\
  f (e^{i \theta} z) = e^{i k \theta} f (z) \bigr\}.
\end{multline*} 

Our proof relies  on the multipole expansion of the logarithm kernel \cite{SteinWeiss1971}*{\S IV.5.7}, which is an identity related to the generating function of the Chebyshev polynomials 
and is also known as the cylindrical multipole expansion (see formula \eqref{eqMultipoleExpansion}).
The corresponding multipole expansion of the Newtonian kernel was already used
in the proof of the nondegeneracy of the groundstate solution for the three-dimensional
Choquard equation, see \cite{Lenzmann2009}.

Finally we emphasize that the nondegeneracy of the groundstate is an important spectral assumption in a series of papers on effective solitary waves motion and semi-classical limit for Hartree type equations (see for instance \citelist{\cite{DavSqua}\cite{BonnanodAveniaGhimentiSquassina201}\cite{weiwinter}}). In a forthcoming paper we use our nondegeneracy result for proving existence result of semiclassical states for the planar Schr\"odinger--Newton system.

\medskip
{\bf Acknowledgements.} The research of the authors was supported by GNAMPA project 2016 ``Studio variazionale di fenomeni fisici non lineari'', the Projet de Recherche (Fonds de la Recherche Scientifique--FNRS) T.1110.14 ``Existence and asymptotic behavior of solutions to systems of semilinear elliptic partial differential equations'', the Mandat d'Impulsion Scientifique (FNRS) F.4508.14 ``Patterns, Phase Transitions, 4NLS \& BIon'' and the ARC AUWB-2012-12/17-ULB1- IAPAS.

\section{Variational framework}\label{sec:symm-uniq-posit}

We begin by showing that the planar Choquard equation \eqref{prob2}
can be derived from the Newton--Schr\"odinger system by a formal inversion.

Let us consider a classical solutions $(u,w)$ of the planar Schr\"odinger--Newton system
\begin{equation}
\label{eq:17}
\left\{
\begin{aligned}
-\Delta u + a u &=w u  &&\qquad \text{in $\R^2$,}\\
-\Delta w &= u^2 &&\qquad \text{in $\R^2$},\\
\end{aligned}
\right.
\end{equation}
where $a$ is positive constant, subject to the conditions
\begin{equation}\label{eq:22-0}
u \in L^\infty(\R^2) \qquad \text{and} \qquad w(x) \to -\infty  \quad \text{as $\abs{x} \to \infty$.}
\end{equation}

By Agmon's Theorem (see \cite{agmon}), \eqref{eq:17} and \eqref{eq:22-0} imply that
\begin{equation}
\label{eq:25}
u(x)= o(e^{-\alpha \abs{x}}) \qquad \text{as $\abs{x} \to \infty$ for every $\alpha >0$.}
\end{equation}
Moreover, since every semibounded harmonic function $\R^2 \to \R$ is
constant, we have
\begin{equation}
\label{eq:26}
w(x)= c + \frac{1}{2 \pi} \int_{\R^2} \ln \frac{1}{\abs{x - y}} \abs{u (y)}^2\dif y, 
\end{equation}
for every $x \in\R^2$ and some constant $c \in \R$.

\medskip
We recall that the solutions for which $u$ is positive are known to enjoy symmetry properties 
\cite{CW}*{Theorem 6.1} up  to the symmetries of the problem. 
Precisely, the following result holds. We write $u_\lambda(\cdot)$ to denote $\lambda^2 u (\lambda \cdot)$ for $\lambda\ne0$.
\begin{theorem} \label{sec:symm-uniq-posit-11}
Let $a>0$. If $(u, w)$ is a classical solution of \eqref{eq:17} and \eqref{eq:22-0} with $u > 0$ in $\R^2$, then, up to translation, the functions $u$ and $w$ are radially symmetric and strictly radially decreasing. 
Moreover, if $(\Tilde{u}, \Tilde{w})$ is another classical  solution of \eqref{eq:17} and \eqref{eq:22-0} with $\Tilde{u} > 0$ in $\R^2$, then the exists $x_0 \in \R^2$ and $\lambda > 0$ such that
for each $x \in \R^2$,
\begin{equation*}
\left\{\begin{array}{rl}\Tilde{u}(x)\!\!\!\!  &= u_\lambda(x-x_0)
\\
\Tilde{w}(x)\!\!\!\!  &= w_{\lambda}(x-x_0)
+ a \bigl(1 - \lambda^2\bigr).
\end{array}\right.
\end{equation*}
\end{theorem}

It follows from Theorem~\ref{sec:symm-uniq-posit-11}, that the solution \( (u, w)\) is unique up to translations, under a \emph{suitable additional condition at infinity} on \(w\). Indeed, we know from 
\eqref{eq:26} that there exists $c \in\R$ such that 
\[
\frac{1}{2 \pi} \int_{\R^2} \ln \frac{1}{\abs{x - y}} \abs{u (y)}^2\dif y -w(x) = c. 
\] 
Therefore, if $\rho>0$ and \((\Tilde{u}, \Tilde{w})\) is another classical solution of \eqref{eq:17} and \eqref{eq:22-0}, then, with \(\lambda > 0\) given by Theorem~\ref{sec:symm-uniq-posit-11},
\begin{multline*}
 \frac{1}{2 \pi} \int_{\R^2} \ln \frac{\rho}{\abs{x - y}} \abs{\Tilde u (y)}^2 \dif y
 - \Tilde w (x)\\
 = \lambda^2 \Bigl(\frac{1}{2 \pi} \int_{\R^2} \ln \frac{1}{\abs{\lambda x - y}} \abs{u (y)}^2 \dif y
 - w (\lambda x)\Bigr) + \frac{\lambda^2\ln \lambda\rho}{2 \pi} \int_{\R^2} \abs{u}^2 + a \bigl(\lambda^2 - 1\bigr)\\
 = \lambda^2 c+ \frac{\lambda^2\ln \lambda\rho}{2 \pi} \int_{\R^2} \abs{u}^2 + a \bigl(\lambda^2 - 1\bigr).
\end{multline*}
Since the right hand side is an increasing continuous function of \(\lambda\) that takes \(-a\) as a limit at \(0\) and diverges to \(+\infty\) at \(+\infty\), there exists a unique solution \(\Tilde{u}\) such that 
\[
 \frac{1}{2 \pi} \int_{\R^2} \ln \frac{\rho}{\abs{x - y}} \abs{\Tilde u (y)}^2 \dif y
 - \Tilde w (x) = 0.
\]
In particular, the asymptotic boundary condition  
$$
 \lim_{x \to \infty}  w (x) - \frac{1}{2 \pi} \int_{\R^2} \ln \frac{\rho}{\abs{x - y}} \abs{u (y)}^2 \dif y
 = 0
$$
implies uniqueness of the solution up to translations. Because of the decay of $u$ at infinity, this is equivalent with requiring the asymptotic condition
$$
  \lim_{x \to \infty} w (x) - \frac{1}{2\pi} \ln \frac{\rho}{\abs{x}} \int_{\R^2} \abs{u}^2 = 0.
$$
Fixing thus \(\rho > 0\), we are reduced to consider the integro-differential equation
\begin{equation}
\label{eq:56bis}
-\Delta u + a u =  \frac{1}{2 \pi}  \Bigl[\ln \frac{1}{|\cdot|} \ast \abs{u}^2 \Bigr] \ u  \qquad \text{in $\R^2$}.
\end{equation}
Solutions of \eqref{eq:56bis} are formally critical points of the functional
\[
 \frac{1}{2} \int_{\R^2} \abs{\nabla u}^2 +a  \abs{u}^2 - \frac{1}{8 \pi} \int_{\R^2} \ln \frac{1}{\abs{x - y}} \abs{u (x)}^2 \abs{u (y)}^2 \dif x \dif y.
\]

The first integral is the norm on the Sobolev space \(H^1 (\R^2)\) induced by the scalar product
\[
  \scalprod{u}{v} = \int_{\R^2} \bigl(\nabla u \cdot \nabla v+ a\,
uv \bigr), \qquad \text{for $u,v \in H^1(\R^2)$},
\]
by 
\(\norm{u}^2:= \scalprod{u}{u}\) for each $u \in H^{1}(\R^2)$.
The second integral is not continuous on \(H^1 (\R^2)\) \citelist{\cite{Stubbe}\cite{CW}}.
However, an adequate functional setting that we now recall has been introduced by J.\thinspace{}Stubbe \cite{Stubbe} who used a smaller space with a stronger norm. 
One first defines for the functions \(f, g : \R^2 \to \R\), the three symmetric bilinear forms
\begin{align*}
 B^+ (f, g) &= \frac{1}{2 \pi} \iint_{\R^2 \times \R^2} \ln_+ \frac{1}{\abs{x - y}} f(x) g (y) \dif x \dif y,\\
 B^- (f, g) &= \frac{1}{2 \pi} \iint_{\R^2 \times \R^2} \ln_+ \abs{x - y} f (x) g (y) \dif x \dif y,\\
 B (f, g) & = \frac{1}{2 \pi} \iint_{\R^2 \times \R^2}\ln \frac{1}{\abs{x - y}} f (x) g (y) \dif x \dif y,
\end{align*}
whenever the integrand is Lebesgue measurable,
so that in particular,
\[
 B (f, g) = B^+ (f, g) - B^- (f, g).
\]
The classical Young convolution inequality, see for example \cite{LiebLoss2001}*{theorem 4.2}, implies
\[
 \abs{B^+ (f, g)}
 \le \frac{1}{2 \pi} \Bigl(\int_{\R^2} \Bigl(\ln_+ \frac{1}{\abs{x}} \Bigr)^\frac{p}{p - 2}\dif x\Bigr)^{2 - \frac{2}{p}}
 \norm{f}_{L^p(\R^2)} \norm{g}_{L^p(\R^2)}
\]
for every \(p \in (2, +\infty)\).
On the other hand, we have for every \(x, y \in \R^2\)
\[
  \ln_+ \abs{x - y} \le \ln_+ (\abs{x} + \abs{y})
  \le \ln_+ \abs{x} + \ln_+ \abs{y},
\]
so that 
\[
\begin{split}
 \abs{B^- (f, g)}
 &= \frac{1}{2 \pi} \Bigabs{\iint_{\R^2 \times \R^2} \ln_+ \abs{x - y} f (x) g (y) \dif x \dif y}\\
 &\le \frac{1}{2 \pi}\left(\norm{g}_{L^1(\R^2)} \int_{\R^2} \abs{f (x)} \ln_+ \abs{x} \dif x 
 + \norm{f}_{L^1(\R^2)}\int_{\R^2} \abs{g (x)} \ln_+ \abs{x} \dif x\right). 
\end{split}
\]

We have thus proved  

\begin{proposition}
For every \(p > 2\), the bilinear form \(B\) is well defined and bounded on \(Y \times Y\), where the space
\[
 Y = \Bigl\{ f : \R^2 \to \R \st \int_{\R^2} \Big(\abs{f(x)}^p + \abs{f(x)} (1 + \ln_+ \abs{x})\Big)\dif x < \infty \Bigr\},
\]
is endowed with the norm defined for \(f \in Y\) by
\[
 \norm{f}_{Y} = \norm{f}_{L^p(\R^2)}
 + \norm{f}_{L^1(\R^2)} + \int_{\R^2} \abs{f (x)}\ln_+ \abs{x} \dif x.
\]
\end{proposition}

In order to go back to our original functional, we first note that the multiplication map 
\((u, v) \mapsto u v\) is a bilinear map which is bounded from \(Z \times Z\) to \(Y\), where
the space \(Z\) is defined by
\[
 Z = \Bigl\{ u : \R^2 \to \R \st \int_{\R^2} \abs{u(x)}^{2p} + (1 + \ln_+ \abs{x})\abs{u(x)}^2 \dif x < \infty \Bigr\}
\]
is endowed with the norm defined for \(u \in Z\) by
\[
 \norm{u}_{Z} = \Bigl(\int_{\R^2} \abs{u}^p \Bigr)^\frac{1}{2 p}
 + \Bigl(\int_{\R^2} (1 + \ln_+ \abs{x})\abs{u (x)}^2  \dif x\Bigr)^\frac{1}{2}.
\]
We now define the functional space 
\[
X:= \Bigl\{u \in H^1(\R^2)\st \int_{\R^2} \ln_+ \abs{x}\, \abs {u (x)}^2\dif x < \infty\Bigr\} \]
endowed with
norm defined through
\[
  \norm{u}_X^2:= \int_{\R^2} \abs{\nabla u(x)}^2 + \bigl(1 + \ln_+ \abs{x}\bigr) \abs{u (x)}^2\dif x,
\]
on which we consider the functional 
\[
  I (u) = \int_{\R^2} (\abs{\nabla u}^2 + a \abs{u}^2) \dif x - \frac{1}{4} B (u^2, u^2). 
\]
Since the second term of the functional \(I\) is the composition of the continuous linear embedding of \(X\) into \(Z\), a continuous bilinear map from \(Z \times Z\) to \(Y\) and a  continuous bilinear map from \(Y \times Y\) to \(\R\), it follows that the functional \(I\) is \emph{smooth} on the space \(X\). Moreover, its first two derivatives are given by
\[
 I' (u)[\varphi] = \int_{\R^2} \bigl(\nabla u \cdot \nabla \ \varphi + a\, u \varphi\bigr) - B (u^2, u \varphi),
\]
and 
\[
 I'' (u)[\varphi, \psi]= \int_{\R^2} \bigl(\nabla \varphi \cdot \nabla \psi + a\, \varphi \psi\bigr) - B (u^2, \varphi \psi)
 - 2 B (u \varphi, u \psi),
\]    
for each \(u, \varphi, \psi \in X\).

\section{Asymptotic behaviour of the groundstate solution}\label{Asymptotic behaviour of the groundstate solution}

The goal of the present section is to study the asymptotics of the groundstate solution to \eqref{prob2} and prove Theorem~\ref{theoremAsymptotics}.

\subsection{Rough asymptotics for linear Schr\"odinger operators with a logarithmic potential}

We first construct upper and lower solutions to a linear problem related to 
\eqref{prob2}. These estimates are too rough to deduce Theorem~\ref{theoremAsymptotics}.
We state them because the proof is quite elementary and might help the reader to understand the more sophisticated construction in the proof of Lemma~\ref{lemmaRefinedUpperAndLowerSolutions} below. 
Moreover, the reader can verify that these rough estimates would be sufficient to obtain the proof of Theorem~\ref{theoremAsymptotics} concerning the nondegeneracy which is given in Section~\ref{sec:nondegeneracy}.

\begin{lemma}
\label{lemmaRoughUpperAndLowerSolutions}
Let \(V \in C (\R^N)\) be such that
\[
 \lim_{\abs{x} \to \infty} \frac{V (x)}{\ln \abs{x}} = \lambda > 0.
\]
For every \(\varepsilon > 0\), 
there exist \(\underline{W}_\varepsilon\), \(\overline{W}_\varepsilon\) and \(R_\varepsilon > 0\) such that 
\begin{gather*}
  -\Delta \underline{W}_\varepsilon + V \underline{W}_\varepsilon \le 0 \qquad \text{in \(\R^2 \setminus B_{R_\varepsilon}\)},\\
  -\Delta \overline{W}_\varepsilon + V \overline{W}_\varepsilon \ge 0\qquad \text{in \(\R^2 \setminus B_{R_\varepsilon}\)},\\
  \lim_{\abs{x} \to \infty} \frac{\underline{W}_\varepsilon (x)}{\exp\, (-(1 + \varepsilon) \abs{x} \sqrt{\lambda \ln \abs{x}}\,)} = 1,\\
  \lim_{\abs{x} \to \infty} \frac{\overline{W}_\varepsilon (x)}{\exp\, (-(1 - \varepsilon) \abs{x} \sqrt{\lambda \ln \abs{x}}\,)} = 1.
\end{gather*}
\end{lemma}

\begin{proof}
We define, for every \(\tau \in \R\), the function 
\[
w_\tau : (1, +\infty)\to \R : r\mapsto  \exp (- \tau r \sqrt{\ln r}).
\]
We compute directly for each \(r \in (1 +\infty)\),
\[
 w_\tau' (r) = -\tau\biggl(\sqrt{\ln r} + \frac{1}{2
\sqrt{\ln r}}\biggr) w_\tau (r)
\]
and 
\[
  w_\tau'' (r) 
  = \tau\Biggl(\tau
    \biggl(\sqrt{\ln r} + 
           \frac{1}{2 \sqrt{\ln r}}\biggr)^2
    - \frac{1}{2 r \sqrt{\ln r}} 
    + \frac{1}{4 r \sqrt{(\ln r)^3}} 
         \Biggr) 
         w_\tau (r).
\]
If we define the function \(W_\tau : \R^2 \setminus B_1\) for each \(x \in \R^2 \setminus B_1\) by \(W_\tau (x) = w_\tau (\abs{x})\), we have
\[
 -\Delta W_\tau (x) + V (x) W_\tau (x)
 =  \biggl(\ln \abs{x} \Bigl(\frac{V (x)}{\ln \abs{x}} - \tau^2 \Bigr)
 + O (1) \biggr) W_\tau (x).
\]
We obtain the conclusion by taking \(R_\varepsilon > 1\) sufficiently large and by choosing \(\underline{W}_{\varepsilon} = W_{\sqrt{\lambda}\,(1 - \varepsilon)}\) and \(\overline{W}_{\varepsilon} = W_{\sqrt{\lambda}\,(1 + \varepsilon)}\).
\end{proof}

We immediately deduce from Lemma \ref{lemmaRoughUpperAndLowerSolutions} some rough asymptotic decay estimates on the solutions of linear equations with a potential growing logarithmically at infinity.

\begin{corollary}
\label{corollaryRoughAsymptotics}
Assume that \(V \in C (\R^N)\) and that there exists $\lambda>0$ such that
\[
 \lim_{\abs{x} \to \infty} \frac{V (x)}{\ln \abs{x}} = \lambda.
\]
If \(u\) is a positive solution of 
\[
 -\Delta u + V u = 0,\quad x\in\R^2,
\]
then
\[
 \lim_{\abs{x} \to \infty} \frac{\ln u (x)}{\abs{x} \sqrt{\ln \abs{x}}} = 
-\sqrt{\lambda}.
\]
\end{corollary}

\subsection{Refined asymptotics for linear Schr\"odinger operators with a logarithmic potential}

In order to obtain fine asymptotics, we rely on the following construction of upper and lower solutions.

\begin{lemma}
\label{lemmaRefinedUpperAndLowerSolutions}      
Assume \(V \in C (\R^N)\) is such that for some \(\beta \in (0, 1]\),
\[
 \frac{V' (r)}{V (r)} = \frac{1}{r \ln r} + o \Bigl(\frac{1}{r^{\beta + 
1}}\Bigr)
\]
as \(r \to \infty\).
Then, for \(R > 0\) sufficiently large,
there exist radial functions \(\underline{W}_\pm \in C^2 (\R^2 \setminus B_R)\) 
and \(\overline{W}_\pm \in C^2 (\R^2 \setminus B_R)\) such that \(\underline{W}_\pm > 0\) 
and \(\overline{W}_\pm > 0\) 
in \(\R^2 \setminus B_R\),
\begin{gather*}
  -\Delta \underline{W}_\pm (x) + V(\abs{x})\underline{W}_\pm(x) \le 0\qquad \text{for \(x \in \R^2 \setminus B_R\)},\\
  -\Delta \overline{W}_\pm(x) + V(\abs{x}) \overline{W}_\pm(x) \ge 0\qquad \text{for \(x \in \R^2 \setminus B_R\)},\\
  \underline{W}_\pm (x) = \bigl(1 + O (\abs{x}^{-\beta})\bigr)\frac{ \exp \Bigl(\pm
\displaystyle \int_R^{\abs{x}} \sqrt{V (s)} \dif s\Bigr)}{\abs{x}^{1/2}\,(\ln 
\abs{x})^{1/4}} 
\\
  \overline{W}_\pm (x) = \bigl(1 + O (\abs{x}^{-\beta})\bigr)\frac{\exp \Bigl(\pm
\displaystyle \int_R^{\abs{x}} \sqrt{V (s)} \dif s\Bigr)}{\abs{x}^{1/2}\,(\ln 
\abs{x})^{1/4} },
\end{gather*}
and 
\[
  \overline{W}_\pm (x)= \underline{W}_\pm (x) \biggl(1 + O \biggl(\frac{1}{\abs{x}^{\beta + 1}}\biggr)\biggr),
\]
as \(\abs{x} \to \infty\).
\end{lemma}

\begin{proof}%
[Proof of Lemma~\ref{lemmaRoughUpperAndLowerSolutions}]
We define  the functions \(w_{\tau, +}\) and \(w_{\tau, -}\) for every \(\tau \in \R\) and \(r \in (0, +\infty)\) by
\[
  w_{\tau, \pm} (r) = \exp \Bigl(\pm \int_0^r \sqrt{V (s)} \dif s\Bigr){r^{-1/2} 
(\ln (r/R))^{-1/4}} \Bigl(1 - \frac{\tau}{r^\beta} \Bigr).
\]
We compute directly
\[
 w_{\tau, \pm}' (r) = \Bigl(\pm\sqrt{V (r)} - \frac{1}{2 r} - \frac{1}{4 r \ln r} + 
\frac{\beta \tau}{r^{\beta + 1} - \tau r}  \Bigr) w_{\tau, \pm} (r)
\]
and 
\[
\begin{split}
  w_{\tau, \pm}'' (r) = &\biggl( \Bigl(\pm\sqrt{V (r)} - \frac{1}{2 r} - \frac{1}{4 
r \ln r} + \frac{\beta \tau}{r^{\beta + 1} - \tau r}  \Bigr)^2\\
  & \pm \frac{V' (r)}{2 \sqrt{V (r)}} + \frac{1}{2 r^2} + \frac{1}{4 r^2 \ln 
r} + \frac{1}{4 r^2 (\ln r)^2} \\
&\hspace{5em}- \frac{\beta \tau \bigl((\beta + 1)r^\beta - 
\tau)}{(r^{\beta + 1} - \tau r)^2}\biggr) w_{\tau, \pm} (r).
\end{split}
\]
If we set \(W_{\tau, \pm} (x) = w_{\tau, \pm} (\abs{x})\), we have
\begin{multline*}
 -\Delta W_{\tau, \pm} (x) + V (x) W_{\tau, \pm} (x)\\
 =  \biggl(\pm\sqrt{V (r)}\Bigl(\frac{V' (r)}{2 V (r)} - \frac{1}{2 r \ln r} + 
\frac{2\beta \tau }{r^{\beta + 1} - \tau r}\Bigr) +  O 
\Bigl(\frac{1}{r^2}\Bigr)\biggr) W_{\tau, \pm} (x).
\end{multline*}
We therefore conclude by taking \(\underline{W}^\pm = w_{\mp 1, \pm}\) and \(\overline{W}^\pm
=w_{\pm 1, \pm}\).
\end{proof}

If \(V (x) = \Tilde{V} (x) \ln \abs{x}\), the assumption of Lemma \ref{lemmaRefinedUpperAndLowerSolutions} can be written
\[
 \frac{\Tilde{V}' (x)}{\Tilde{V} (x)} = o \Bigl(\frac{1}{\abs{x}^{\beta + 1}} 
\Bigr)
\]
as \(\abs{x} \to \infty\).
As we derived  Corollary~\ref{corollaryRoughAsymptotics} from Lemma~\ref{lemmaRoughUpperAndLowerSolutions}, we are able to improve the asymptotics 
of the solutions of linear equations thanks to Lemma~\ref{lemmaRefinedUpperAndLowerSolutions}.

\begin{corollary}
\label{corollaryRefinedAsymptotics}
Assume \(V \in C (\R^N)\) is such that for some \(\beta \in (0, 1]\),
\[
 \frac{V' (\abs{x})}{V (\abs{x})} = \frac{1}{\abs{x} \ln \abs{x}} + o 
\Bigl(\frac{1}{\abs{x}^{\beta + 1}}\Bigr)
\]
as \(\abs{x} \to \infty\). Let \(u\) be a positive radial solution of
\[
 -\Delta u + V u = 0.
\]
If
\[
  u (x) = o \left(\frac{\exp \Bigl(\int_0^{\abs{x}} \sqrt{V (s)} \dif s \Bigr)}{\abs{x}^{1/2} (\ln \abs{x})^{1/4}}\right),
\]
then there exists \(\mu \in (0, +\infty)\) such that, as \(\abs{x} \to \infty\),
\[
 u (x) = \frac{\bigl(\mu + O\bigl(r^{-\beta}\bigr)\bigr)\exp \Bigl(- \displaystyle \int_0^{\abs{x}} \sqrt{V (s)} \dif s\Bigr)}{\abs{x}^{1/2} (\ln \abs{x})^{1/4}}.
\]
\end{corollary} 
\begin{proof}
Let \(R > 0\), \(\underline{W}_\pm\) and \(\overline{W}_\pm\) be given by Lemma~\ref{lemmaRefinedUpperAndLowerSolutions}.
We now consider the function 
\[
  v_{\varepsilon, r}
  = (1 + \varepsilon) \frac{u (r)}{\underline{W}_- (r)} \underline{W}_- 
  - \varepsilon \frac{u (r)}{\overline{W}_+ (r)} \overline{W}_+ - u, 
\]
for every $r\ge R$.
By the growth assumptions on \(u\) and by the known growth of \(\overline{W}_-\) and \(\underline{W}_+\),
the set 
\[
  \Omega_{\varepsilon, r} = \bigl\{ x \in \R^2 \setminus B_{r} \st v_{\varepsilon, r} (x) > 0\bigr\}
\]
is bounded.
Moreover
\[
 -\Delta v_{\varepsilon, r} + V v_{\varepsilon, r} \le 0,
\]
and \(v_{\varepsilon,r} \le 0\) on \(\partial B_{r}\),
so that we infer from the weak maximum principle for second order linear operators on bounded sets that
\(
 v_{\varepsilon, r} \le 0
\)
in \(\Omega_{\varepsilon,r}\). Henceforth, we conclude that \(v_{\varepsilon,r} \le 0\) in \(\R^2 \setminus B_{r}\).
Since \(\varepsilon > 0\) is arbitrary, we have 
\[
 u (s) \ge \underline{W}_{-}(s) \frac{u (r)}{\underline{W}_- (r)}
\]
for all \(s \ge r \ge R\). Arguing similarly with the function 
\[
 w_{\varepsilon, r} 
  = (1 + \varepsilon) u - \frac{u (r)}{\overline{W}_- (r)} \overline{W}_- 
  - \varepsilon \frac{u (r)}{\underline{W}_+ (r)} \underline{W}_+,
\]
we deduce that 
\[
 u (s) \le \overline{W}_- (s) \frac{u (r)}{\overline{W}_- (r)},
\]
for all \(s \ge r \ge R\).
By the asymptotic equivalence of \(\underline{W}_-\) and \(\overline{W}_-\), we have thus proved
\[
  1
  \le 
  \frac{u (s) /\underline{W}_- (s)}{u (r) /\underline{W}_- (r)}
  \le 1 + O \Bigl(\frac{1}{r^{\beta + 1}}\Bigr),
\]
and it follows that the function \(u / \underline{W}_-\) has a limit at infinity
by the Cauchy criterion of convergence.
\end{proof}

\subsection{Asymptotics on the logarithmic potential}

In order to use our previous asymptotic estimates, we need to understand the logarithmic term \(\ln \ast \abs{u}^2\). A na\"ive approach gives the following inequality. 

\begin{proposition}
\label{propositionVDecay}
If \(f\) is radial and \(f\in L^1 (\R^2)\), then for each \(x \in \R^2\)
\[
 \Bigabs{\int_{\R^2} \ln \abs{x - y} \, f (y) \dif y - \ln \abs{x} \int_{\R^2} f (y)\dif y} 
 \le \int_{\R^2 \setminus B_{\abs{x}}} \ln \frac{\abs{y}}{\abs{x}} \abs{f (y)} \dif y.
\]
\end{proposition}

\begin{proof}
Since \(f\) is a radial function, we obtain by Newton's shell theorem, see for instance \cite{LiebLoss2001}*{Theorem 9.7}, for each \(x \in \R^2\),
\[
\begin{split}
 \int_{\R^2} \ln \abs{x - y} \, f (y) \dif y 
 & =\ln \abs{x} \int_{B_{\abs{x}}} f (y) \dif y + \int_{\R^2 \setminus B_{\abs{x}}} \ln \abs{y}\, f (y) \dif y\\
 &=\ln \abs{x} \int_{\R^2} f (y) \dif y 
 + \int_{\R^2 \setminus B_{\abs{x}}} \ln \frac{\abs{y}}{\abs{x}} f (y) \dif y.\qedhere
\end{split}
\]
\end{proof}

\subsection{Asymptotics for the groundstate}

We now go back to the logarithmic Choquard problem \eqref{prob2} for which we derive the sharp asymptotics at infinity of groundstates, that is, Theorem~\ref{theoremAsymptotics}.

\begin{proof}[Proof of Theorem~\ref{theoremAsymptotics}]
We first observe that since \(\lim_{\abs{x} \to \infty} (\frac{\ln}{2\pi} \ast \abs{u}^2) (x) = - \infty\), we have for each \(\lambda > 0\), 
\[
 -\Delta u + \lambda u \le 0,\quad x\in \R^2 \setminus B_R
\]
for some \(R > 0\) large enough depending on \(\lambda\).
It follows therefrom that 
\[
 \lim_{\abs{x} \to \infty} e^{\lambda \abs{x}} u (x) = 0
\]
for every \(\lambda > 0\). 
In view of Proposition~\ref{propositionVDecay}, this leads to 
\[
 \lim_{\abs{x} \to \infty} \Bigl(w (x) + M \ln \abs{x}\Bigr) e^{\lambda \abs{x}} = 0,
\]
where $w$ has been defined in \eqref{eq:w}
 and 
\[
 M = \frac{1}{2 \pi} \int_{\R^2} \abs{u}^2.
\]
This implies that, as \(r \to \infty\),
\begin{multline*}
 \int_{e^{-a/M}}^r \sqrt{a - w (s)} \dif s \\
 = \int_{e^{-a/M}}^r \sqrt{a + M \ln s} \dif s + \int_{e^{-a/M}}^\infty \left(\sqrt{a - w (s)} - \sqrt{a + M \ln (s)} \right)\dif s+ o (1),
\end{multline*}
where the second integral on the right-hand side is finite and does not depend on the variable \(r\).
We observe now by the elementary change of variable \(s = e^{-a/M} t\) that 
\[
\begin{split}
 \int_{e^{-a/M}}^r \sqrt{a + M \ln s} \dif s
 = \sqrt{M} e^{-a/M} \int_{1}^{e^{a/M} r} \sqrt{\ln t} \dif t.
\end{split}
\]
The asymptotic estimate now follows from the linear asymptotics of Corollary~\ref{corollaryRefinedAsymptotics}.
\end{proof}

\begin{remark}
\label{remarkSpecialFunctions}
The integral appearing in the statement of Theorem~\ref{theoremAsymptotics} can be expressed in terms of classical special functions. Indeed, by integration by parts and by a change of variable \(s = \exp (\sigma^2)\),
\[
\begin{split}
 \int_1^\lambda \sqrt{\ln s} \dif s 
 &= \lambda \sqrt{\ln \lambda} - \int_1^\lambda \frac{1}{2 \sqrt {\ln s}} \dif s = \lambda \sqrt{\ln \lambda} - \int_0^{\sqrt{\ln \lambda}} e^{\sigma^2} \dif \sigma\\
 & = \lambda \bigl(\sqrt{\ln \lambda} - F (\sqrt{\ln \lambda})\bigr) = \lambda \sqrt{\ln \lambda} - \frac{\sqrt{\pi}}{2} \operatorname{erfi} \bigl(\sqrt{\ln \lambda}\bigr),\\
 & = \lambda \sqrt{\ln \lambda} - \frac{\gamma \bigl( \tfrac{1}{2}, - \ln \lambda\bigr)}{2(-1)^{1/2}}  = \frac{\gamma \bigl(\frac{3}{2}, -\ln \lambda\bigr)}{(-1)^{3/2}}\\
 &= \sqrt{\ln \lambda}\, \Gamma (\tfrac{3}{2}) \gamma^* (\tfrac{3}{2},-\ln \lambda)\,.
\end{split}
\]
where \(F\) is \emph{Dawson's integral}, \(\operatorname{erfi}\) is the \emph{imaginary error function} (defined for \(z \in \C\) by \(\operatorname{erfi} (z) = - i\operatorname{erfi} (z i)\)) and \(\gamma\) is the \emph{lower incomplete gamma function} (using the same branch in its computation than for \((-1)^{1/2}\)) \cite{Temme2010} and \(\gamma^*\) is its entire part \cite{Paris2010}.
\end{remark}

\section{Nondegeneracy of the positive solutions}%
\label{sec:nondegeneracy}

Let \(u \in X\) be a solution of the planar logarithmic Choquard equation \eqref{prob2}
which does not change sign and satisfies the variational characterization
\[
  I(u) = \inf_{u \in X} \sup_{t \in \R} I(tu).
\]
Let \(w: \R^2 \to \R\) be the function defined for each \(x \in \R^2\) by
\[
 w (x) = \frac{1}{2\pi} \int_{\R^2} \ln \frac{1}{\abs{x - y}} \abs{u (y)}^2 \dif y .
\]
By Theorem~\ref{sec:symm-uniq-posit-11}, up to translation, this solution \(u\) is unique and radially symmetric.
We can therefore assume without loss of generality that the functions \(u\) and \(w\) are \emph{radially symmetric}.
We consider the linear operator \(\mathcal{L}(u)\) defined by 
\eqref{eq:L(u)}, that is, 
\[
  \mathcal{L}(u) \varphi = -\Delta \varphi + (a - w) \varphi + 2 u \Bigl( \frac{\ln}{2\pi} \ast (u \varphi)\Bigr)
\]
on the space $\tilde X$ defined by \eqref{eq:tildeX}. In the sequel, we just write $\mathcal L$ to shorten the notation. 

\subsection{Closedness of the operator $\mathcal{L}$}
We show that the operator \(\mathcal{L}\) is closed.

\begin{proposition}
\label{propositionClosedness}
The operator \(\mathcal{L}\) is a closed operator from \(L^2 (\R^2)\) to \(L^2 (\R^2)\).
\end{proposition}

The proof will rely on the following estimate.

\begin{lemma}
\label{lemmaInteractionTerm}
For each \(\varepsilon > 0\), there exists \(C_\varepsilon > 0\) such that for each \(\varphi, \psi \in H^1 (\R^2)\),
\begin{multline*}
 \Bigabs{\int_{\R^2} \int_{\R^2} \ln \abs{x - y} u (x) \varphi (x) u (y) \psi (y)\dif x\dif y}\\
 \le \varepsilon \Bigl(\int_{\R^2} \abs{\nabla \varphi}^2\Bigr)^\frac{1}{2}\Bigl(\int_{\R^2} \abs{\nabla \psi}^2\Bigr)^\frac{1}{2}
  + C_\varepsilon \Bigl(\int_{\R^2} \abs{\varphi}^2\Bigr)^\frac{1}{2}\Bigl(\int_{\R^2} \abs{\psi}^2\Bigr)^\frac{1}{2}.
\end{multline*}
\end{lemma}
\begin{proof}
Let \(\delta > 0\), and let \(A_\delta = \{(x, y) \in \R^2 \times \R^2 \st \abs{x - y} \le \delta \}\).
Using the Young convolution inequality and the Sobolev inequality for some fixed \(p \in (2, +\infty)\), we have then
\begin{multline*}
 \Bigabs{\int_{A_\delta} \ln \abs{x - y} u (x) \varphi (x) u (y) \psi (y)\dif x\dif y}\\
 \le \Bigl(\int_{|z|\le\delta } \abs{\ln \abs{z}}^{\frac{p}{2 (p - 1)}} \Bigr)^{2 - \frac{2}{p}}
 \Bigl(\int_{\R^2} \abs{u}^p \abs{\varphi}^p \Bigr)^\frac{1}{2}\Bigl(\int_{\R^2} \abs{u}^p \abs{\psi}^p \Bigr)^\frac{1}{2}\\
 \le \varepsilon \Bigl(\int_{\R^2} \abs{\nabla \varphi}^2 + \abs{\varphi}^2 \Bigr)^\frac{1}{2}\Bigl(\int_{\R^2} \abs{\nabla \psi}^2 + \abs{\psi}^2 \Bigr)^\frac{1}{2},
\end{multline*}
provided that \(\delta>0\) is sufficiently small.
On the other hand, we observe that if \(\abs{x - y}\ge \delta\), \(\abs{x} \ge \delta\) and \(\abs{y} \ge \delta\), then 
\[
 0\le \ln \frac{\abs{x - y}}{\delta} \le \ln_+ \frac{\abs{x}}{\delta} + \ln_+ \frac{\abs{y}}{\delta}
\]
so that, because of the exponential decay of \(u\) (Theorem~\ref{theoremAsymptotics}),
\begin{multline*}
 \Bigabs{\int_{\R^2 \times \R^2 \setminus A_\delta} \ln \abs{x - y} u (x) \varphi (x) u (y) \psi (y)\dif x\dif y}\\
 \le \int_{\R^2 \times \R^2} \Bigl(\abs{\ln \delta} + \ln_+ \frac{\abs{x}}{\delta} + \ln_+ \frac{\abs{y}}{\delta}\Bigr)\abs{u (x)}\,\abs{u (y)}\, \abs{\varphi (x)}\,\abs{\psi (y)}\dif x \dif y\\
 \le C \Bigl(\int_{\R^2} \abs{\varphi}^2\Bigr)^\frac{1}{2}\Bigl(\int_{\R^2} \abs{\psi}^2\Bigr)^\frac{1}{2}.\qedhere
\end{multline*}
\end{proof}

\begin{proof}%
[Proof of Proposition~\ref{propositionClosedness}]%
Let \((\varphi)_{n \in \N}\) be a sequence in \(\Tilde{X}\) that converges strongly in \(L^2 (\R^2)\) to \(\varphi \in L^2 (\R^2)\) and such that \((\mathcal{L}\varphi_n)_{n \in \N}\) converges strongly in \(L^2 (\R^2)\) to some \(f \in L^2 (\R^2)\).
We observe that, in view of Lemma~\ref{lemmaInteractionTerm}, for each \(n \in \N\),
\[
 \int_{\R^2} \varphi_n \mathcal{L} \varphi_n
 \ge \frac{1}{2} \int_{\R^2} \abs{\nabla \varphi}^2 + \int_{\R^2} (- w) \abs{\varphi}^2 - C \int_{\R^2}  \abs{\varphi}^2.
\]
By the convergences of the sequences \((\varphi_n)_{n \in \N}\) and \((\mathcal{L} \varphi_n)_{n \in \N}\) and by the asymptotic behaviour of the function \(w\), it follows that the sequence \((\varphi_n)_{n \in \N}\) is bounded in \(X\), and thus it converges weakly to \(\varphi\) in the space \(X\).

If we now take \(\psi \in C^1_c (\R^2)\), we have 
\[
 \int_{\R^2} \psi f = \lim_{n \to \infty} \int_{\R^2} \psi \mathcal{L} \varphi_n
 = \lim_{n \to \infty} \int_{\R^2} \varphi_n \mathcal{L} \psi 
 = \int_{\R^2} \varphi \mathcal{L} \psi.
\]
Since \(\varphi \in X\), we have by the H\"older inequality and by the exponential decay of \(u\) we have for \(x \in \R^2\) large enough, \(\abs{\ln \ast (u \varphi) (x)} \le C \ln \abs{x}\). In particular \(u \,(\ln \ast (u\varphi)) \in L^2 (\R^2)\). 
\end{proof}

\subsection{Angular splitting of the operator $\mathcal{L}$}

Since the solution \(u\) is radial, 
the operator \(\mathcal{L}\) commutes with rotations acting on \(L^2 (\R^2)\).
This suggests to use the orthogonal splitting \cite{SteinWeiss1971}*{\S IV.2}
\[
 L^2 (\R^2; \C) = \bigoplus_{k \in \Z} L^2_k (\R^2; \C),
\]
where the summands 
\begin{multline*}
 L^2_k (\R^2; \C) = \bigl\{ f \in L^2 (\R^2;\C) \st \text{for almost every \(z \in \R^2 \simeq \C\)
  and \(\theta \in \R\), }\\
  f (e^{i \theta} z) = e^{i k \theta} f (z) \bigr\},
\end{multline*}
are closed mutually orthogonal subspaces of \(L^2 (\R^2; \C)\).
In particular \(L^2_0 (\R^2; \C)\) is the subspace of radial functions of \(L^2 (\R^2; \C)\).
In general, if \(f \in L^2_k (\R^2; \C)\), then there exists 
a function \(g : (0, +\infty) \to \C\) 
such that 
\[
 \int_0^\infty \abs{g (r)}^2 r \dif r < \infty
\]
and for each \(r \in (0, +\infty)\) and \(\theta \in \R\)
\[
 f (r e^{i\theta}) = g (r) e^{i k \theta}.
\]
In order to describe how the linear operator \(\mathcal{L}\) acts on each of the subspaces \(L^2_k (\R^2; \C)\),
we rely on the multipole expansion of the logarithm kernel \cite{SteinWeiss1971}*{\S IV.5.7}: 
for each \(x = r e^{i \theta} \in \R^2 \simeq \C\) and \(y = s e^{i \eta} \in \R^2 \simeq \C\), 
if \(r > s\), then 
\begin{equation}
\label{eqMultipoleExpansion}
 \ln \abs{x - y}
 =\ln r + \sum_{k = 1}^\infty \Bigl(\frac{s}{r}\Bigr)^k \frac{\cos (k (\theta - \eta))}{k}.
\end{equation}
This identity is related to the generating function of the Chebyshev polynomials 
and is also known as the cylindrical multipole expansion.
The formula \eqref{eqMultipoleExpansion} follows directly from the convergence 
of the Taylor series of the complex logarithm, that is, 
\[
 \ln \abs{x - y} = \ln r + \operatorname{Re} \Bigl(\ln \Bigl(1 - \frac{s}{r} e^{i (\eta - \theta)}\Bigr)\Bigr)
 = \ln r + \sum_{k = 1}^\infty  \Bigl(\frac{s}{r}\Bigr)^k \frac{\cos (k (\theta - \eta))}{k} .
\]
The corresponding multipole expansion of the Newtonian kernel was used
in the proof of the nondegeneracy of the groundstate solution for the three-dimensional
Choquard equation \cite{Lenzmann2009}.

\smallbreak

If \(\varphi \in \Tilde{X} \cap L^2_k (\R^2; \C)\), then we can write \(\varphi (r e^{i \theta})
= \psi (r) e^{i k \theta}\) for some \(\psi : (0, +\infty) \to \C\), and 
\[
 (\mathcal{L} \varphi) (r e^{i \theta}) = 
 (\mathcal{L}_k \psi) (r) e^{i k \theta},
\]
where for each \(k \in \Z \setminus \{0\}\), the operator \(\mathcal{L}_k\) is defined by
\begin{multline*}
 \mathcal{L}_k \psi (r)
 = -\psi'' (r) - \frac{1}{r} \psi' (r)
 + \Bigl( a + \frac{k^2}{r^2} - w (r)\Bigr)\psi (r) \\
  - \frac{u (r)}{\abs{k}}  \int_0^\infty  \psi (s)  u(s) \Bigl(\frac{\min (r, s)}{\max (r, s)} \Bigr)^k 
  s \dif s,
\end{multline*}
whereas for \(k = 0\), the operator \(\mathcal{L}_0\) is defined by
\begin{multline*}
 \mathcal{L}_0 \psi (r)
 = -\psi'' (r) - \frac{1}{r} \psi' (r)
  + \bigl(a - w (r)\bigr) \psi (r) \\
  - 2 u (r) \int_0^\infty  
  \psi (s) u(s)  \ln \frac{1}{\min (r, s)} s \dif s.
\end{multline*}
This last formula can also be obtained by Newton's shell theorem. We also have
\[
  w (r) = \int_0^\infty  \abs{u(s)}^2  s \ln \frac{1}{\max (r, s)} \dif s.
\]
Observe in particular that if \(\varphi \in \Tilde{X} \cap L^2_k (\R^2; \C)\), 
we have \(\mathcal{L} \varphi \in L^2_k (\R^2; \C)\) and therefore 
\[
 \ker \mathcal{L} = \bigoplus_{k \in \Z} \bigl(\ker \mathcal{L} \cap L^2_k (\R^2; \C)\bigr).
\]
This allows to study separately the kernels of the operators \(\mathcal{L}_k\) which is our aim in the next subsections.

\subsection{Radial eigenfunctions}
We show that the kernel of the operator \(\mathcal{L}_0\) is trivial.
As in \cite{Lenzmann2009}, we first decompose the operator \(\mathcal{L}_0\) as follows
\[
 \mathcal{L}_0 \psi (r)
 = \Hat{\mathcal{L}}_0 \psi (r) + 2 u (r) \int_0^\infty  u (s) \psi (s)\, s \ln s\dif s,
\]
where the operator \(\Hat{\mathcal{L}}_0\) is defined by
\begin{equation}
\label{eqLHat}
  \Hat{\mathcal{L}}_0 \psi (r)
  = -\psi'' (r) - \frac{1}{r} \psi' (r) + \bigl(a - w (r)\bigr)\psi (r) 
  + 2 u (r) \int_0^r u (s) \psi (s)\, s \ln \frac{r}{s} \dif s
\end{equation}
and we prove the exponential growth of solutions \(v\) to the linear equation 
\(\Hat{\mathcal L}_0 v =0\).

\begin{lemma}
\label{decay}
If \(\psi \in C^2 ([0, +\infty);\C)\) satisfies \(\Hat{\mathcal{L}}_0 \psi = 0\), then, either 
\(\psi = 0\) or there exists \(C_1, C_2 > 0\) such that 
\[
  \abs{\psi (r)} \ge C_1 u (r) \exp \Bigl(C_2 \int_1^r \frac{1}{s \abs{u (s)}^2} \dif s\Bigr),
\]
for each \(r \ge 1\). In this last case, we have \(\lim_{r \to \infty} \abs{\psi (r)} = \infty\).
\end{lemma}
  
\begin{proof}

%
%
%
As \(u\) is a solution of \eqref{prob2}, it satisfies for each \(r \in (0, +\infty)\),
\begin{equation}
\label{eqUradial}
-u''(r) - \frac{1}{r} u'(r) + \bigl(a - w (r)\bigr) u(r) = 0.
\end{equation}
Since $\Hat{\mathcal L}_0 \psi =0$, it follows from \eqref{eqUradial} that for every \(r \in (0, +\infty)\),
\[
\begin{split}
 \frac{d}{d r} \bigl(r (u (r) \psi' (r) &- u' (r) \psi (r))\bigr)\\
 &= r \Bigl( u (r) \psi'' (r) + \frac{1}{r} u (r) \psi' (r) - u'' (r) \psi (r) - \frac{1}{r} u' (r) \psi (r)\Bigr)\\
 &= 2 r \abs{u (r)}^2 \int_0^r \psi (s) u (s)\, s \ln \frac{r}{s} \dif s.
\end{split}
\]
Integrating, this implies that if \(\eta = \psi /u\),
\[
\begin{split}
 \eta' (r)&=
 \frac{r \bigl(u (r) \psi' (r) - \psi (r) u' (r)\bigr)}{r \abs{u (r)}^2}\\
 &= 2 \int_0^r \frac{s \abs{u (s)}^2} {r \abs{u (r)}^2}\int_0^s \eta (t) \abs{u (t)}^2 \, t \ln \frac{s}{t}\dif t \dif s.
\end{split}
\]
Integrating again, we finally obtain, by exchanging the order of integration
\begin{equation}
\label{eqIntegralEqeta}
\begin{split}
  \eta (r) &=
  \eta (0)
  + 2\int_0^r 
  \int_0^s \frac{t \abs{u (t)}^2}{s \abs{u (s)}^2} \int_0^t \eta (q) \,q \abs{u (q)}^2 \ln \frac{t}{q}\dif q \dif t \dif s\\
  &= \eta (0)
  + \int_0^r \eta (s) b (s) \dif s
\end{split}
\end{equation}
where the function \(b : (0, +\infty) \to \R\) is defined for \(s \in (0, +\infty)\) by
\[
  b (s) = 2 \int_{s}^r \int_{t}^r 
  \frac{s \abs{u (s)}^2 \, t \abs{u (t)}^2}{q \abs{u (q)}^2} \ln \frac{t}{s} \dif q \dif t.
\]
The formula \eqref{eqIntegralEqeta} is in fact the integral form of a first-order homogeneous linear differential equation.
Such an equation has an explicit solution, given by 
\begin{equation}
\begin{split}
 \eta (r)
 &=\eta (0) \exp \Bigl(\int_0^r b (s) \dif s\Bigr)\\
 &= \eta (0) \exp 
 \Bigl(2 \int_0^r \int_{s}^r \int_{t}^r 
  \frac{s \abs{u (s)}^2 \,t \abs{u (t)}^2}{q \abs{u (q)}^2} 
  \ln \frac{t}{s}  \dif q \dif t \dif s \Bigr).
\end{split}
\end{equation}
We observe now that if \(r \ge 1\),
\begin{multline*}
 \int_0^r \int_{s}^r \int_{t}^r 
  \frac{s \abs{u (s)}^2 \,t \abs{u (t)}^2}{q \abs{u (q)}^2} 
  \ln \frac{t}{s}  \dif q \dif t \dif s\\
  \ge \int_0^1 \int_{s}^1 s \abs{u (s)}^2 \,t \abs{u (t)}^2 \ln \frac{t}{s} \dif t \dif s 
  \int_{1}^r \frac{1}{q \abs{u (q)}^2} \dif q\\
  \ge c \int_1^r \frac{1}{q \abs{u (q)}^2} \dif q,
\end{multline*}
from which the conclusion follows.
\end{proof}     

To go on, in view of \eqref{eqUradial}, we compute
\[
 \Hat{\mathcal{L}}_0 u (r) = 2 u (r) \int_0^r \abs{u (s)}^2\, s \ln \frac{r}{s} \dif s,
\]
for each \(r \in (0, +\infty)\). If we now set \(z (r) = r u' (r)\), we get 
\begin{multline*}
 \Hat{\mathcal{L}}_0 z (r) 
 = - r u''' (r)- 3 u'' (r) - \frac{1}{r} u' (r)\\
 + \bigl(a - w (r)\bigr) r u' (r) + 2 u(r) \int_0^r u (s) u' (s)\, s^2 \ln \frac{r}{s} \dif s.
\end{multline*}
On the other hand, by differentiating the equation \eqref{eqUradial} satisfied by \(u\), we have, for each \(r \in (0, +\infty)\)
\begin{equation}
\label{eqUradialdifferentiated}
  - u''' (r) - \frac{1}{r} u''(r) + \frac{1}{r^2} u' (r)
  + \bigl(a - w (r)\bigr) u' (r) - w' (r) u (r) = 0,
\end{equation}
where 
\begin{equation}
\label{eqVradialdifferentiated}
 w' (r) = -\int_0^r \abs{u(s)}^2 \frac{s}{r} \dif s. 
\end{equation}
Combining \eqref{eqUradial} and \eqref{eqUradialdifferentiated}, we deduce that
\[
\begin{split}
  \Hat{\mathcal{L}}_0 z (r) &= u (r) \Bigl(-\int_0^r \abs{u (s)}^2\, s \dif s - 2 a + 2 \int_0^\infty \abs{u (s)}^2\, s \ln \frac{1}{\max (r, s)} \dif s\\
  &\qquad\qquad
  - 2 \int_0^r \abs{u (s)}^2\, s \ln \frac{r}{s} \dif s + \int_0^r \abs{u (s)}^2\, s \dif s\Bigr)\\
  &= u (r) \Bigl(- 2 a + 2 \int_0^\infty \abs{u (s)}^2\, s \ln \frac{1}{s} \dif s
  - 4 \int_0^r \abs{u (s)}^2\, s \ln \frac{r}{s} \dif s\Bigr).
\end{split}
\]
If we now define the function \(\zeta : (0, +\infty) \to \R\) for each \(r \in (0, +\infty)\) by 
\[
 \zeta (r) = 2 u (r) + z (r) = 2 u (r) + r u' (r),
\]
we obtain
\[
 \Hat{\mathcal{L}}_0 \zeta (r) = - 2 \Bigl( a + \int_0^\infty \abs{u (s)}^2\, s \ln s \dif s\Bigr)u (r) .
\]
We claim that 
\[
  a + \int_0^\infty \abs{u (s)}^2\, s \ln s \dif s \ne 0.
\]
Otherwise, since by integration by parts, we have for each \(r \in (0, +\infty)\),
\[
  \int_0^r \zeta (s) \dif s = \int_0^r 2 u (s) + s u' (s) \dif s
  = r u (r) + \int_0^r u (s) \dif s,
\]
by the decay properties of \(u\) and by Lemma~\ref{decay}, this would only be possible if \(\zeta= 0\) on \((0, +\infty)\). This in turn would imply that for each \(r \in (0, +\infty)\), \(u (r) = u (1)/r^2\), which is impossible since the function \(u\) also has to satisfy the equation \eqref{eqUradial}.

We are now in position to prove that there are no radial eigenfunctions.

\begin{lemma}
\label{lemmakerL^2_0}
We have 
$$\ker \mathcal{L} \cap L^2_0 (\R^2; \C)= \ker \mathcal{L}_0 =\{0\}.$$
\end{lemma}
\begin{proof}
Let us assume that \(\varphi \in L^2_0 (\R^2; \C)\setminus\{ 0\}\) and
\[
 \mathcal{L}_0 \varphi = 0.
\]
We then deduce that 
\[
  \Hat{\mathcal{L}}_0 \varphi  = - 2 \Bigl(\int_0^\infty  u (s) \varphi (s)\, s \ln s\dif s\Bigr) u
\]
We define 
\[
  \psi  = \varphi  - \zeta  \frac{\displaystyle \int_0^\infty u (s) \varphi (s)\, s \ln s \dif s }{\displaystyle \int_0^\infty \abs{u (s)}^2\, s \ln s \dif s + a}.
\]
By construction, \(\psi \in L^2_0 (\R^2; \C)\).
By Lemma~\ref{decay}, this implies \(\psi (r) = 0\) for each \(r \in (0, +\infty)\), that is, for every \(r \in (0, +\infty)\),
\[
 \varphi (r) =  \zeta (r)  \frac{\displaystyle \int_0^{\vphantom{1}\infty} u (s) \varphi (s)\, s \ln s \dif s }{\displaystyle \int_0^{\vphantom{1}\infty} \abs{u (s)}^2\, s \ln s \dif s + a}.
\]
Integrating, we infer that 
\[
 \int_0^\infty u (s) \varphi (s)\, s \ln s \dif s = \int_0^{\vphantom{1}\infty}u (s) \zeta (s)\, s \ln s \dif s\frac{\displaystyle \int_0^{\vphantom{1}\infty} u (s) \varphi (s)\, s \ln s \dif s }{\displaystyle \int_0^\infty \abs{u (s)}^2\, s \ln s \dif s + a}
\]
and thus, by the definition of \(\zeta\)
\[
\begin{split}
 \int_0^\infty \abs{u (s)}^2\, s \ln s \dif s + a  &= \int_0^\infty 2 \abs{u (s)}^2\, s \ln s \dif s + \int_0^\infty u (s) u'(s)\, s^2 \ln s \dif s\\
 &=\int_0^\infty \abs{u (s)}^2\, s \ln s \dif s - \frac{1}{2} \int_0^\infty \abs{u (s)}^2\, s \dif s,
\end{split}
\]
which is impossible since \(a > 0\).
\end{proof}

\subsection{Nonradial eigenfunctions}

For every \(k \in \Z \setminus\{0\}\), we have the variational definition of the eigenvalues
\[
 \lambda_0 \bigl(\mathcal{L}_k\bigr)
 = \inf \Bigl\{\mathcal{Q}_k (\psi) \st \psi e^{i k \theta} \in X \text{ and }\int_{0}^\infty \abs{\psi (r)}^2 \, r \dif r = 1 \Bigr\},
\]
where 
\begin{multline*}
 \mathcal{Q}_k (\psi)
 =\int_0^\infty \bigl(\abs{\psi' (r)}^2 + (1 - w (r)) \abs{\psi (r)}^2\bigr) \, r \dif r
 + k^2 \int_0^\infty \frac{\abs{\psi (r)}^2}{r}\dif r\\
 - \operatorname{Re} \int_0^\infty \int_0^\infty  \psi (r) \overline{\psi (s)} K_k (r, s) \dif s \dif r,
\end{multline*}
with
\[
 K_k (r, s) = \frac{1}{\abs{k}} u (r) u(s) \Bigl(\frac{\min (r, s)}{\max (r, s)} \Bigr)^k 
  r s,
\]
is defined for \(\psi \in W^{1, 1}_{\mathrm{loc}} ((0, +\infty); \C)\) such that
\[
 \int_0^\infty \bigl(\abs{\psi' (r)}^2 + \abs{\psi (r)}^2\bigr) \, r + \frac{\abs{\psi (r)}^2}{r} + \abs{\psi (r)}^2 r \ln_+ (r) \dif r < \infty .
\]
Thanks to the logarithmic weight in the definition of the functional space \(X\) and the logarithmic growth of \(w\) at infinity, the eigenvalue \(\lambda_0 (\mathcal{L}_k)\) is achieved.
We now observe that 
\[
 \mathcal{Q}_k (\abs{\psi})
 =\mathcal{Q}_k (\psi) - \int_0^\infty \int_0^\infty \bigl( \abs{\psi (r)}\abs{\psi (s)} - \operatorname{Re} (\psi (r) \overline{\psi (s)})\bigr)\, K_{k}(r, s) \dif s,
\]
Since \(K_k (r, s) > 0\) on \((0, +\infty)\times(0, +\infty)\), it follows that 
\[
 \mathcal{Q}_k (\abs{\psi}) \ge \mathcal{Q}_k (\psi),
\]
with equality if and only if \(\operatorname{Re} (\psi (r) \overline{\psi (s)}) = \abs{\psi (s)}\abs{\psi (r)}\) for almost every \((r, s) \in (0, +\infty)\times(0, +\infty)\), which implies that there exists \(\alpha \in \C\) such that \(\abs{\alpha} = 1\) and for each \(r \in (0, +\infty)\),
\[
 \psi (r) = \alpha \abs{\psi (r)}.
\]
In particular, the space of eigenvectors corresponding to \(\lambda_0 (\mathcal{L}_k)\)
is spanned by nonnegative eigenvectors. If \(\psi_1\) and \(\psi_2\) are two nonnegative eigenvectors, then the quantity \(D : (0, +\infty)^2 \to \R\)
\[
  D (r, s) = \psi_1 (r) \psi_2 (s) - \psi_2 (r) \psi_1 (s)
\]
has constant sign, since \(D (\cdot, s)\) and \(D (r, \cdot)\) are real eigenvectors corresponding to the eigenvalue \(\lambda_0 (\mathcal{L}_k)\). Since \(D (r, s) = - D (s, r)\), it follows that \(D (r, s ) = 0\). As a consequence, we deduce that \(\lambda_0 (\mathcal{L}_k)\) is a simple eigenvalue and we can assume that the associated eigenfunction is a nonnegative function \(\psi_k : (0, +\infty) \to \R\).

Note that if \(\abs{\ell} \le \abs{k}\), then \(K_k \le K_\ell\) pointwize. 
Therefore, we have, since \(\psi_k \ge 0\), 
\[
\begin{split}
 \mathcal{Q}_k (\psi_k)
 &=\mathcal{Q}_\ell (\psi_k)
 + \bigl(k^2  - \ell^2 \bigr)\int_0^\infty \frac{\abs{\psi_k(r)}^2}{r}\dif r\\
 & \qquad+ \int_0^\infty \psi_k (r) \psi_k (s)\bigl(K_\ell (r, s)- K_k (r, s)\bigr)\dif s \dif r\\
 &> \mathcal{Q}_\ell (\psi_k),
\end{split}
\]
from which it follows that 
\(
 \lambda_0 (\mathcal{L}_k) > \lambda_0 (\mathcal{L}_{\ell}). 
\)
We can now conclude our proof of Theorem \ref{theoremNondegenerate} by showing that \(\mathcal{\lambda}_0 (\mathcal{L}_1) = 0\).

\begin{proof}[Proof of Theorem \ref{theoremNondegenerate}]
We observe that 
\begin{multline*}
 \mathcal{L}_1 (u') (r)=
 -u''' (r) - \frac{u'' (r)}{r} + \frac{u' (r)}{r^2} + \bigl(a - w (r)\bigr) u' (r)\\
 - u (r) \int_0^\infty  u (s)  u(s) \frac{\min (r, s)}{\max (r, s)} 
  s \dif s.
\end{multline*}
By \eqref{eqUradialdifferentiated} and by integration by parts, we deduce that
\[
 \mathcal{L}_1 (u') (r)=  - u (r) \Bigl(\int_0^r \abs{u(s)}^2 \frac{s}{r} \dif s +  \int_0^\infty  u' (s)  u(s) \frac{\min (r, s)}{\max (r, s)} s \dif s\Bigr)
 = 0.
\]
This implies that \(u'\) is an eigenfunction of the operator \(\mathcal{L}_1\). Assume by contradiction that \(\lambda_0 (\mathcal{L}_1) < 0\). Then by orthogonality of eigenfunctions, we have
\[
 \int_0^\infty \psi_1 (r) u' (r) \, r \dif r = 0,
\]
which cannot hold since the function \(u\) is decreasing and the function \(\psi_1\) is nonnegative. Therefore we have shown \(\mathcal{\lambda}_0 (\mathcal{L}_1) = 0\).
Since we already proved that the eigenvalue \(\lambda_0 (\mathcal{L}_1)\) is simple, it follows that \(\ker \mathcal{L}_{1} = \operatorname{span} \langle u'\rangle\). By the discussion preceding the proof, we also know that \(\ker \mathcal{L}_{k}=\{0\}\) whenever \(\abs{k} > 1\) and Lemma \ref{lemmakerL^2_0} implies \(\ker \mathcal{L}_{0}=\{0\}\).
We have thus proved that 
\[
\ker \mathcal{L} = \bigoplus_{k \in \Z} \bigl(\ker \mathcal{L}_{k}\bigr) =  \mathcal{L}_{1}=\{u'(r)e^{i \theta} : \theta \in[0,2\pi)\}.\qedhere
\] 
\end{proof}

\begin{remark}
Alternatively, we could have used a nonlocal groundstate representation formula \cite{MorozVanSchaftingen2012}*{Proposition 2.1}, to write
\[
 \mathcal{Q}_1 (\psi) = \int_0^\infty (\psi/u')' r \dif r
 + \frac{1}{2} \int_0^\infty K_1 (r, s) u' (r) u' (s) \Bigl(\frac{\psi (r)}{u' (r)} - \frac{\psi (s)}{u' (s)} \Bigr)^2 \dif r \dif s
\]
and obtain that the eigenvalue \(\mathcal{\lambda}_0 (\mathcal{L}_1)\) is equal to \(0\) and is simple. 
\end{remark}

\end{document}